\documentclass[reqno]{amsart}

\usepackage{amsmath,amsthm,amssymb}
\usepackage{mathrsfs}
\usepackage{dsfont}
\usepackage{color}
\usepackage{enumerate}
\usepackage{cases,subeqnarray}

\usepackage{cleveref}
\crefname{section}{\textsl{Section}}{section}
\crefname{subsection}{\textsl{Subsection}}{subsection}
\crefname{definition}{\textsl{Definition}}{definition}
\crefname{theorem}{\textsl{Theorem}}{theorem}
\crefname{lemma}{\textsl{Lemma}}{lemma}
\crefname{proposition}{\textsl{Proposition}}{proposition}
\crefname{corollary}{\textsl{Corollary}}{corollary}
\crefname{remark}{\textsl{Remark}}{remark}
\crefname{equation}{}{}

\allowdisplaybreaks[3]

\newtheorem{lemma}{Lemma}[section]
\newtheorem{theorem}{Theorem}[section]
\newtheorem{corollary}{Corollary}[section]

\newtheorem{remark}{Remark}[section]

\numberwithin{equation}{section}

\DeclareMathOperator*{\osc}{osc}

\newcommand{\abs}[1]{\left\vert#1\right\vert}
\newcommand{\norm}[1]{\left\Vert#1\right\Vert}
\newcommand{\set}[1]{\left\{#1\right\}}

\newcommand{\R}{\mathbb R}
\newcommand{\C}{\mathbb C}
\newcommand{\Z}{\mathbb Z}
\newcommand{\N}{\mathbb N}

\newcommand{\CC}{\mathcal C}

\newcommand{\CS}{\mathcal S}

\newcommand{\fc}{\mathfrak c}
\newcommand{\fs}{\mathfrak s}

\newcommand{\ve}{\varepsilon}
\newcommand{\vth}{\vartheta}

\newcommand{\ol}{\bar}
\newcommand{\ul}{\underline}
\newcommand{\p}{\partial}
\newcommand{\dps}{\displaystyle}
\newcommand{\wt}{\widetilde}
\newcommand{\wh}{\widehat}
\newcommand{\ra}{\rightarrow}

\begin{document}

\title[A Bernstein problem for special Lagrangian equations]
{A Bernstein problem for special Lagrangian equations in exterior domains}

\author{Dongsheng Li}
\address{Dongsheng Li: School of Mathematics and Statistics,
Xi'an Jiaotong University, Xi'an 710049, China;}
\email{\tt lidsh@mail.xjtu.edu.cn}

\author{Zhisu Li}
\address{Zhisu Li: School of Mathematics and Statistics,
Xi'an Jiaotong University, Xi'an 710049, China;}
\email{\tt lizhisu@stu.xjtu.edu.cn}

\author{Yu Yuan}
\address{Yu Yuan:
Department of Mathematics,
University of Washington,
Box 354350; Seattle, WA 98195-4350; USA;}
\email{\tt yuan@math.washington.edu}

\thanks{Zhisu Li is the corresponding author.}

\date{\today}


\begin{abstract}
We establish quadratic asymptotics for solutions
to special Lagrangian equations with supercritical phases in exterior domains.
The method is based on an exterior Liouville type result
for general fully nonlinear elliptic equations
toward constant asymptotics of bounded Hessian,
and also certain rotation arguments toward Hessian bound.
Our unified approach also leads to quadratic asymptotics for convex solutions
to Monge-Amp\`{e}re equations (previously known), quadratic Hessian equations,
and inverse harmonic Hessian equations over exterior domains.
\end{abstract}

%

\maketitle


\section{Introduction}

In this paper, we establish an exterior Bernstein type result
for special Lagrangian equations with supercritical phases:
every exterior solution is asymptotic to a quadratic polynomial at infinity.
\begin{theorem}\label{thm.u=Q+O.SLE}
Let $u$ be a smooth solution of the special Lagrangian equation
\begin{equation}\label{eqn.SLE}
\sum_{i=1}^{n}\arctan\lambda_i(D^2u)=\Theta
\quad\mathrm{in}\quad\R^n\setminus\ol{\Omega},
\end{equation}
where constant $\Theta$ satisfies $|\Theta|>(n-2)\pi/2$,
$\Omega$ is a bounded domain,
and $\lambda_i(D^2u)$'s denote the eigenvalues of the Hessian $D^2u$.
Then there exists a unique quadratic polynomial $Q(x)$
such that when $n\geq3$,
\begin{equation}\label{eqn.SLE.u=Q+ON}
u(x)=Q(x)+O_k(|x|^{2-n})\quad\mathrm{as}\quad|x|\ra\infty
\end{equation}
for all $k\in\N$, and when $n=2$,
\begin{equation}\label{eqn.SLE.u=Q+log+ON}
u(x)=Q(x)+\frac{d}{2}\log{x^T(I+(D^2Q)^2)x}+O_k(|x|^{-1})
\quad\mathrm{as}\quad|x|\ra\infty
\end{equation}
for all $k\in\N$, where
\[d=\frac{1}{2\pi}\left(\int_{\p\Omega}\cos\Theta\,u_{\nu}
+\sin\Theta\,u_{1}(u_{22},-u_{12})\cdot\nu\,ds-\sin\Theta\,|\Omega|\right),\]
$\nu$ is the unit outward normal of the boundary $\p\Omega$,
and the notation $\varphi(x)=O_k(|x|^m)$
means that $|D^{k}\varphi(x)|=O(|x|^{m-k})$.
\end{theorem}


Special Lagrangian equation \eqref{eqn.SLE}
is the potential equation for the minimal Lagrangian or ``gradient'' graph
$(x,Du(x))\subset\R^n\times\R^n$, in calibrated geometry \cite{HL78}.
When $n=2$, the trigonometric equation \eqref{eqn.SLE}
also takes the algebraic form
$\cos\Theta\,\Delta u+\sin\Theta\,\det D^2u=\sin\Theta\,$;
while for $n=3$, and $|\Theta|=\pi$ or $\pi/2$,
equation \eqref{eqn.SLE} is equivalent to $\Delta u=\det D^2u$ or
$\sigma_2(D^2u)=\lambda_1\lambda_2+\lambda_2\lambda_3+\lambda_3\lambda_1=1$
respectively.
The phase or Lagrangian angle $(n-2)\pi/2$ is called \textit{critical},
since the level set
$\set{\lambda\in\R^n\,|\,\text{$\lambda$ satisfying \eqref{eqn.SLE}}}$
is convex \textit{only} when $|\Theta|\geq(n-2)\pi/2$ \cite[Lemma 2.1]{Y06}.
Simple solution $\sin x_1e^{x_2}$ and precious one
$(x_1^2+x_2^2)e^{x_3}-e^{x_3}+e^{-x_3}/4$ \cite{W16} to
\eqref{eqn.SLE} with $\Theta=(n-2)\pi/2$, $n=2$ and $n=3$
respectively show that the ``critical'' phase condition
in \cref{thm.u=Q+O.SLE} is indeed necessary.
The ``entire'' Bernstein-Liouville type problem has been well-studied,
see for instance \cite{Bo92,Fu98,BCGJ03,Y02,Y06,WY08}.

Corresponding to minimal surface equations over exterior domains,
there are the well-known exterior Bernstein type results
only in low dimensions \cite{B51} (for $n=2$) \cite{Si87} (for $3\leq n\leq7$),
which assert that all solutions approach to linear functions asymptotically near infinity.
The same linear asymptotics continue to hold in all higher dimensions,
if certain necessary conditions
such as the boundedness of the gradient of solutions are assumed (cf. \cite{Si87}).
For Monge-Amp\`{e}re equations in exterior domains, there are exterior
J\"{o}rgens-Calabi-Pogorelov type results \cite{FMM99} (for $n=2$)
and \cite{CL03} \cite{BLZ15} (for $n\geq2$),
which state that all (convex) solutions are asymptotic to
quadratic polynomials near infinity.

Heuristically the plane asymptotic behavior for minimal surfaces in exterior domains
(quadratic for special Lagrangian equations and linear for minimal surface systems)
is ``clear''---seen through the monotonicity formula---%
once the tangent cone at infinity is flat.
But as one tries to employ Allard's $\ve$-regularity to locate the ``flat'' plane,
those approximated planes over larger and larger annuli
could potentially keep changing.
This difficulty still prevents us from seeing the quadratic asymptotics
for solutions to general special Lagrangian equations
and linear asymptotics for solutions to minimal surface systems in exterior domains,
where entire rigidity results are available, or tangent cones at infinity are flat,
for example \cite{Y02,WY08} (except for convex solutions) and \cite{HJW80,Fi80,X03}.

To circumvent the difficulty in proving \cref{thm.u=Q+O.SLE},
we take advantage of the fully nonlinear elliptic equation
with concavity satisfied by the single potential.
The key is to show that the Hessian of the solutions has a finite limit at infinity.
Still unlike in the case of minimal surface equation,
the gradient of the solution enjoys Moser's Harnack inequality,
then the limit of the bounded gradient can be quickly drawn at infinity.
Fortunately, the pure second derivatives of the solutions
are supersolutions to the linearized elliptic equation,
then satisfy Krylov-Safonov's weak Harnack inequality (over annuli)
and Evans-Kylov's Hessian estimates.
From here, the limit of the Hessian at infinity can be achieved.
This is the content of \cref{sec.ELT.FNL},
where a finer exterior Liouville theorem
for general fully nonlinear uniformly elliptic concave equations
with bounded Hessian (\cref{thm.u=Q+O.FNL})---%
along with an exterior Liouville theorem for positive solutions
to linear elliptic equations in nondivergence form
(\cref{thm.ELT-LEE-ND})---is established.

There is still another hurdle in making all the above work:
we need the Hessian of solutions to be bounded
and the fully nonlinear concave equation to be uniformly elliptic.
This is done via a rotation device developed in \cite{Y02,Y06};
see the proof of \cref{thm.u=Q+O.SLE} in \cref{sec.EBT.SLE}.

In passing, we make the following remarks.
All solutions to special Lagrangian equation \eqref{eqn.SLE}
with critical phase and with quadratic growth near infinity
must have the same quadratic asymptotic behavior,
if one combines the a priori gradient and Hessian estimates
in \cite{WY09a,WY10,WdY14} with our general exterior Liouville \cref{thm.u=Q+O.FNL}.
This exterior Liouville type result and the exterior Bernstein \cref{thm.u=Q+O.SLE}
also hold true for continuous viscosity solutions,
in light of the regularity for solutions of special Lagrangian equations
\cite{WY09a,WY09b,WY10,CWY09,WdY14}.

We close this introduction by the following. It turns out that our arguments toward \cref{thm.u=Q+O.SLE}
also lead to quadratic asymptotics for convex solutions
to quadratic Hessian equations and inverse harmonic Hessian equations in exterior domains,
for which Chang and the third author \cite{CY10} and Flanders \cite{Fl60}
obtained entire Liouville type results respectively.
Our unified approach also gives a different proof
for the J\"{o}rgens-Calabi-Pogorelov type result
for Monge-Amp\`{e}re equations in exterior domains, which, as mentioned above,
have been done earlier by Ferrer-Mart\'{\i}nez-Mil\'{a}n \cite{FMM99}
and Caffarelli-Li \cite{CL03}; see \cref{sec.FP}.

\section{Exterior Liouville theorems}\label{sec.ELT.FNL}

In this section, we establish the following Liouville type theorem
for general fully nonlinear elliptic equations
with bounded Hessian in exterior domains.
\begin{theorem}\label{thm.u=Q+O.FNL}
Let $u$ be a smooth solution of
\begin{equation}\label{eqn.FNL}
F(D^2u)=0
\end{equation}
in the exterior domain $\R^n\setminus\ol B_1$,
where $n\geq3$, $F$ is uniformly elliptic
with the ellipticity constants $\lambda$ and $\Lambda$,
and also, $F$ is either convex, or concave,
or the level set $\set{M\,|\,F(M)=0}$ is convex.
Suppose
\[\norm{D^2u}_{L^\infty\left(\R^n\setminus\ol B_1\right)}\leq K<+\infty.\]
Then there exists a unique quadratic polynomial
\[Q(x)=\frac{1}{2}x^{T}Ax+b^Tx+c\]
such that \[u(x)=Q(x)+O(|x|^{2-n})\quad\mathrm{as}\quad|x|\ra\infty.\]
Furthermore, if $F$ is infinitely smooth, then we have
\[u(x)=Q(x)+O_k(|x|^{2-n})\quad\mathrm{as}\quad|x|\ra\infty\]
for all $k\in\N$.
\end{theorem}

By symmetry, we discuss only the case that $F$ is concave,
which implies that the pure second derivative $u_{ee}$,
for any fixed direction $e\in\p B_1$,
is a subsolution of the linearized equation
$F_{M_{ij}}(D^2u(x))v_{ij}=0$ of \eqref{eqn.FNL}.
Before going further, we first collect here some preliminary results,
for their proofs one may consult \cite{CC95} and \cite{GT98}.
\begin{enumerate}[\quad(1)]
\item
(Krylov-Safonov's weak Harnack inequality in annulus)
Let $v$ be a nonnegative supersolution of $a^{ij}(x)v_{ij}=0$
in $B_{(1+3\gamma)R}\setminus\ol B_{(1-3\gamma)R}$,
where $a^{ij}(x)$ is uniformly elliptic
with the ellipticity constants $\lambda$ and $\Lambda$,
and $0<\gamma<1/3$ is a constant. Then
\[\left(\frac{1}{|B_{\gamma R}|}\int_{B_{(1+\gamma)R}\setminus
\ol B_{(1-\gamma)R}}v^\delta\right)^{1/\delta}
\leq C\inf_{B_{(1+\gamma)R}\setminus\ol B_{(1-\gamma)R}}v,\]
where $\delta=\delta(n,\lambda,\Lambda)>0$
and $C=C(n,\lambda,\Lambda,\gamma)>0$.

\item
(Evans-Krylov estimate)
Let $u$ be a solution of \cref{eqn.FNL}
and $u_{ee}$ be its pure second derivative
in any fixed direction $e\in\p B_1$.
Then there exist $C=C(n,\lambda,\Lambda)>0$
and $\alpha=\alpha(n,\lambda,\Lambda)>0$, such that
\[\osc_{B_r(z)}u_{ee}\leq\osc_{B_r(z)}D^2u
\leq C\left(\frac{r}{R}\right)^\alpha\osc_{B_R(z)}D^2u
\leq 2CK\left(\frac{r}{R}\right)^\alpha\]
for any $0<r<R$ and any $B_R(z)\subset\R^n\setminus\ol B_1$.
\end{enumerate}

\subsection{Limit of the Hessian}\label{subsec.D2u->A.FNL}
\quad

The key step toward \cref{thm.u=Q+O.FNL} is the following lemma.
\begin{lemma}\label{lem.D2u->A.FNL}
Let $u$ be as in \cref{thm.u=Q+O.FNL}.
Then there exists a symmetric matrix $A$ such that
\[D^2u(x)\ra A\quad\mathrm{as}\quad|x|\ra\infty.\]
\end{lemma}

To prove this, we need only to show that,
for any fixed $e\in\p B_1$,
the pure second derivative $u_{ee}$
tends to some constant number at infinity.

\begin{proof}[Proof of \cref{lem.D2u->A.FNL}]
Set $w(x)=u_{ee}(x)$, $\dps\ol w=\varlimsup_{|x|\ra\infty}w(x)$
and $\dps\ul w=\varliminf_{|x|\ra\infty}w(x)$ for convenience.
It suffices to prove that
\[\ol w=\ul w.\]
If it were not, we would have $\ol w-\ul w=:5d>0$.
Clearly, for any $0<\ve<d$,
there exists a large constant $R=R(\ve)>1$ such that
$\ul w-\ve\leq w(x)\leq\ol w+\ve$ for all $x\in B_{R/2}^\CC$,
and also there exists a sequence of $\ul x_k$ in $B_{R/2}^\CC$,
tending to $\infty$, such that
\[w(\ul x_k)\leq\ul w+\ve\]
for all $k\in\Z^+$.
Then there exists a point $\ol x$ on the sphere $\p B_{|\ul x|}$
for at least one $\ul x\in\set{\ul x_k}$, such that
\[w(\ol x)\geq\ol w-\ve.\]
Otherwise, as a subsolution, $w<\ol w-\ve$
on the spheres $\p B_{|\ul x_k|}$
for all $k\in\Z^+$, by comparison principle,
we would have $w(x)<\ol w-\ve$ for all $x\in B_{|\ul x_1|}^\CC$,
which leads to $\ol w<\ol w-\ve$, a contradiction.

Applying the Evans-Krylov estimate
to $w=u_{ee}$ in $B_{|\ul x|-1}(\ul x)$,
we obtain
\[\osc_{B_{\gamma|\ul x|}(\ul x)}u_{ee}
\leq C\left(\frac{\gamma|\ul x|}{|\ul x|-1}\right)
^\alpha\osc_{B_{|\ul x|-1}(\ul x)}D^2u
\leq4CK\gamma^\alpha\leq d,\]
where $\dps\gamma=\gamma(n,\lambda,\Lambda,K,d)
=:\min\set{1/6,({d}/(4CK))^{1/\alpha}}$.
Thus we deduce that
\[w(x)\leq\ul w+\ve+d\leq\ol w-3d
 \quad\mathrm{or}\quad
 \ol w-w(x)\geq3d 
\quad\mathrm{for}\quad
x\in B_{\gamma|\ul x|}(\ul x).\]
Employing the weak Harnack inequality
to the nonnegative supersolution $v(x)=\ol w+\ve-w(x)$ in the annulus
$B_{(1+3\gamma)|\ul x|}\setminus\ol B_{(1-3\gamma)|\ul x|}$,
we obtain
\[\left(\frac{1}{|B_{\gamma|\ul x|}|}
\int_{B_{(1+\gamma)|\ul x|}\setminus
\ol B_{(1-\gamma)|\ul x|}}v^\delta\right)^{1/\delta}
\leq C\inf_{B_{(1+\gamma)|\ul x|}
\setminus\ol B_{(1-\gamma)|\ul x|}}v\leq Cv(\ol x)\leq 2C\ve.\]
Then $3d\leq 2C\ve,$ where $C$ is independent of $\ve.$
Letting $\ve\ra0$, we get $d=0$, a contradiction.
\end{proof}

\subsection{Finer asymptotic behavior}\label{subsec.u=Q+O.FNL}
\quad

Once the second order derivatives of $u$
in \cref{thm.u=Q+O.FNL} have limits at infinity,
we can get the asymptotic behavior for all other order derivatives of $u$.
To this end, we first note that auxiliary functions
$|x|^{-n}$, and $|x|^{-1/2}$ as well as $|x|^{2-n}-|x|^{2-n-\ve}$
are indeed subsolution and supersolutions, respectively,
to the linearized equations of $F(D^2u)=0$,
which now are close to constant coefficient ones,
say the Laplace equation, near infinity.

\medskip

Next we prove an exterior Liouville theorem for positive solutions
to linear elliptic equations in nondivergence form.
\begin{theorem}\label{thm.ELT-LEE-ND}
Let $v$ be a positive solution
of $a^{ij}(x)v_{ij}=0$ in $\R^n\setminus\ol B_1$,
where $n\geq3$, $a^{ij}(x)$ is uniformly elliptic and
$a^{ij}(x)\ra a_\infty^{ij}$ as $|x|\ra\infty$.
Then there exists a constant $v_\infty$ such that
\begin{equation}\label{eqn.n-2-d.v}
v(x)=v_\infty+o\left(\frac{1}{|x|^{n-2-\delta}}\right)
~\text{as}~|x|\ra\infty,~\text{for all}~\delta>0.
\end{equation}
Furthermore, if we have in addition
\begin{equation}\label{eqn.aij-Holder}
|a^{ij}(x)-a_\infty^{ij}|\leq\frac{C}{|x|^\alpha}
\quad(x\in\R^n\setminus B_1)
\end{equation}
for some positive constants $C$ and $\alpha$, then
\begin{equation}\label{eqn.n-2.v}
v(x)=v_\infty+O\left(\frac{1}{|x|^{n-2}}\right)~\text{as}~|x|\ra\infty.
\end{equation}
\end{theorem}

\begin{proof}
Without loss of generality and for simplicity of notations,
we assume that $a_\infty^{ij}=\delta_{ij}$
and, say $|a^{ij}(x)-\delta_{ij}|\leq1/4$ for $|x|\geq1$.
Note also that the constants $C$'s in the following steps
might be different from line to line.

\medskip

\emph{Step 1.}
We prove $\dps\lim_{|x|\ra\infty}v(x)$
exists and is finite in this step.
Let $\dps\ol v=\varlimsup_{|x|\ra\infty}v(x)$
and $\dps\ul v=\varliminf_{|x|\ra\infty}v(x)$.
Clearly, $\ol v\geq\ul v\geq0$.

We first prove that $\ul v<+\infty$.
Otherwise, we would have
\begin{equation}\label{eqn.v->+8}
v(x)\ra+\infty\quad\mathrm{as}\quad|x|\ra\infty.
\end{equation}
Relying on the equation $a^{ij}(x)w_{ij}=0$,
let us take this solution $v(x)$ and supersolution
$\xi(x)=2|x|^{-1/2}-1$ to bound subsolution $\eta(x)=|x|^{-n}$.
For any $\ve>0$, according to \eqref{eqn.v->+8},
there exists $R_{\ve}>16$ such that
$\ve v(x)>2$ for all $x$ with $|x|\geq R_{\ve}$.
Then $\eta\leq\xi+\ve v$ on $\p B_{R_\ve}\cup\p B_1$.
By the comparison principle, we obtain
$\eta\leq\xi+\ve v$ in $ B_{R_\ve}\setminus B_1$.
In particular, at $x^\ast=(16,0,...,0)$,
\[0<\eta(x^\ast)\leq\xi(x^\ast)+\ve v(x^\ast)=-1/2+\ve v(x^\ast).\]
Letting $\ve\ra0+$, we get $0\leq-1/2$, a contradiction.

Now we prove that $\ol v\leq\ul v$.
For any $\ve>0$, there exists $R_\ve>0$ such that
$\wt v(x)=v(x)-\ul v+\ve>0$ for all $x\in B_{R_\ve}^\CC$.
Since $\dps\varliminf_{|x|\ra\infty}\wt v(x)=\ve$,
there exist $\set{x_k}_{k=1}^{\infty}$ such that
$2R_\ve\leq r_k=|x_k|\ra+\infty$, $r_k<r_{k+1}$ and $\wt v(x_k)\leq2\ve$.
Applying the Krylov-Safonov's Harnack inequality to $\wt v$,
we obtain $\wt v(x)\leq C\wt v(x_k)\leq 2C\ve$
for all $x\in\p B_{r_k}$ and all $k\in\Z^+$.
By the comparison principle, we have
$\wt v(x)\leq 2C\ve$ for all $x\in B_{r_1}^\CC$.
By letting $|x|\ra\infty$ and taking limit superior, we get
$\ol v-\ul v+\ve\leq2C\ve$ for any $\ve>0$.
Letting $\ve\ra0$, we obtain $\ol v\leq\ul v$.

Therefore, $v(x)$ tends to some finite constant $v_\infty$ as $|x|\ra\infty$.

\medskip

To obtain the finer asymptotic behavior
\eqref{eqn.n-2-d.v} and \eqref{eqn.n-2.v},
we follow the arguments of \cite[pp. 324--325]{GS56}
in the rest of the proof.

\medskip

\emph{Step 2.}
Let $\wt\delta=\min\set{\delta/2,(n-2)/2}$.
Consider the supersolution $\phi(x)=|x|^{2-n+\wt\delta}$
to $a^{ij}(x)\phi_{ij}=0$ in $\ol B_{R_\delta}^\CC$.
Since, for any $\ve>0$, there exists $R_\ve>1$,
depending on $\ve$ and $v$, such that
\[|v(x)-v_\infty|<\ve/2,~x\in B_{R_\ve}^\CC,\]
we conclude that there exists $C>0$,
depending on $v$, but independent of $\ve$, such that
\[v(x)-v_\infty\leq C\phi+\ve,
~x\in\p B_{R}\cup\p B_{R_\delta},~R>\max\set{R_\delta,R_\ve}.\]
Applying the comparison principle, we get
\[v(x)-v_\infty\leq C\phi+\ve,~x\in B_{R}\setminus B_{R_\delta}.\]
By letting first $R\ra+\infty$ and then $\ve\ra0+$, we deduce that
\[v(x)-v_\infty\leq C\phi\leq C|x|^{2-n+\wt\delta},~x\in B_{R_\delta}^\CC.\]

Similarly, by considering $v_\infty-v(x)$, we have
\[v(x)-v_\infty\geq-C\phi\geq-C|x|^{2-n+\wt\delta},~x\in B_{R_\delta}^\CC.\]
In summary, the assertion \eqref{eqn.n-2-d.v} is proved.

\medskip

\emph{Step 3.}
In light of the H\"{o}lder continuity condition \eqref{eqn.aij-Holder}
for the coefficient $a^{ij}(x)$ at infinity, as noted in the beginning,
function $\wt\phi(x)=|x|^{2-n}-|x|^{2-n-\alpha/2}$
is a suersolution of $a^{ij}(x)\wt\phi_{ij}=0$ in $\ol B_{R_\alpha}^\CC$.
By taking $\wt\phi$ in place of $\phi$
and following the same lines as in \emph{Step 2},
we conclude that
\[|v(x)-v_\infty|\leq C\wt\phi\leq C|x|^{2-n},~x\in B_{R_\alpha}^\CC,\]
the optimal asymptotic behavior \eqref{eqn.n-2.v}.
This finishes the proof of the lemma.
\end{proof}

\begin{corollary}\label{cor.limit-finite.Dv}
Let $v$ be a smooth solution
of $a^{ij}(x)v_{ij}=0$ in $\R^n\setminus\ol B_1$,
where $n\geq3$, $a^{ij}(x)$ is uniformly elliptic and
$a^{ij}(x)\ra a_\infty^{ij}$ as $|x|\ra\infty$.
Suppose $|Dv(x)|=O(|x|^{-1})$ as $|x|\ra\infty$.
Then there exists a constant $v_\infty$ such that
\begin{equation}\label{eqn.n-2-d.v-Dv}
v(x)=v_\infty+o\left(\frac{1}{|x|^{n-2-\delta}}\right)
~\text{as}~|x|\ra\infty,~\text{for all}~\delta>0.
\end{equation}
Furthermore, if we have in addition
\[|a^{ij}(x)-a_\infty^{ij}|\leq\frac{C}{|x|^\alpha}
\quad(x\in\R^n\setminus B_1)\]
for some positive constants $C$ and $\alpha$, then
\begin{equation}\label{eqn.n-2.v-Dv}
v(x)=v_\infty+O\left(\frac{1}{|x|^{n-2}}\right)~\text{as}~|x|\ra\infty.
\end{equation}
\end{corollary}

\begin{proof}
In virtue of \cref{thm.ELT-LEE-ND},
we need only to show that $v$ is bounded at least on one side.

We show this by contradiction
and by following the same way as in the first part
of the proof of \cite[Theorem 4]{GS56}.
Indeed, if $v$ were unbounded on both sides,
there would exist a sequence $\set{x_k}_{k=1}^{\infty}$,
such that $1<|x_k|<|x_{k+1}|\ra+\infty$
and $v(x_k)=0$ for all $k\in\Z^+$.
Then, it follows from
$|Dv(x)|\leq C/|x|$ (for all $x\in B_1^\CC$) that,
for any $k\in\Z^+$ and any $x\in\p B_{|x_k|}$, we have
\[|v(x)|=\left|\int_{\gamma}\frac{dv}{ds}ds\right|
\leq \frac{C}{|x_k|}\cdot2\pi|x_k|=2C\pi,\]
where the integration path $\gamma$ is the minor arc
connecting $x_k$ and $x$ in the great circle of the sphere $\p B_{|x_k|}$.
By the maximum principle, we thus conclude that
$|v(x)|\leq2C\pi$ on $\ol B_{|x_{k+1}|}\setminus B_{|x_k|}$
for all $k\in\Z^+$.
Therefore, $|v(x)|\leq2C\pi$ on $B_{|x_1|}^\CC$,
contradicts the unboundedness assumption.
\end{proof}

\begin{remark}
\cref{cor.limit-finite.Dv} is slightly different from \cite[Theorem 4]{GS56}
which only asserts that the limit $\dps\lim_{|x|\ra\infty}v(x)$ exists,
but does not state that the limit can not be infinity.
We also note that \cite[Lemma 11]{Se65} says that,
for solution to the linear elliptic equation
in divergence form on the exterior domain,
if the limit exists and the dimension $n\geq3$,
then the limit must be finite.
This result does not need the coefficient $a^{ij}(x)$ converges,
that is ``close to the Laplacian'',
but needs the divergence structure of the equation.
For nondivergence equations,
if the coefficient $a^{ij}(x)$ does not converge,
there are counterexamples for the finiteness of the limit;
for example, function $v(x)=\log|x|$
satisfies the nondivergence uniformly elliptic equation
$(\delta_{ij}+(n-2)x_ix_j|x|^{-2})v_{ij}=0$ in $\R^n\setminus\set{0}$.
Furthermore, $v(x)=\log|x|$ also satisfies
a fully nonlinear concave elliptic equation
$F(D^2v)=(n-2)\lambda_{\mathrm{min}}(D^2v)+\Delta v=0$,
which shows that neither $F(D^2v)=0$
nor its linearized equations have divergence structure for $n\geq3$.
\end{remark}

\medskip

Now we proceed with the proof of \cref{thm.u=Q+O.FNL}.
Note that the constants $C>0$ appeared in the following proof
might be different from line to line.
\begin{proof}[\textbf{Proof of \cref{thm.u=Q+O.FNL}}]
\emph{Step 1.}
Let \[v(x)=u(x)-\frac{1}{2}x^TAx,\]
where $A$ comes from \cref{lem.D2u->A.FNL} and satisfies $F(A)=0$.
Then we have
\[F(D^2v+A)=F(D^2u)=0=F(A),\]
\[\ol a^{ij}v_{ij}
=\int_0^1F_{M_{ij}}(tD^2v(x)+A)dt\cdot\p v_{ij}(x)
=F(D^2v+A)-F(A)
=0,\]
\[\wh a^{ij}(v_e)_{ij}
=F_{M_{ij}}(D^2v(x)+A)(v_e)_{ij}=0,\]
and
\begin{eqnarray*}
\wh a^{ij}(v_{ee})_{ij}
&=&F_{M_{ij}}(D^2v(x)+A)(v_{ee})_{ij}\\
&=&-F_{M_{ij},M_{kl}}(D^2v(x)+A)(v_e)_{ij}(v_e)_{kl}\geq 0,
\end{eqnarray*}
for all $e\in\p B_1$, where
\[\ol a^{ij}(x)=\int_0^1F_{M_{ij}}(tD^2v(x)+A)dt
\quad\mathrm{and}\quad
\wh a^{ij}(x)=F_{M_{ij}}(D^2v(x)+A)\]
are uniformly elliptic with the corresponding ellipticity constants
depending only on $n,\lambda$ and $\Lambda$.

\medskip

It is clear that
\[\ol a^{ij}(x)\ra\ol a_\infty^{ij}=F_{M_{ij}}(A)
\quad\mathrm{and}\quad
\wh a^{ij}(x)\ra\wh a_\infty^{ij}=F_{M_{ij}}(A),\]
since $D^2v(x)\ra0$ ($|x|\ra\infty$)
according to \cref{lem.D2u->A.FNL}.
Thus, by assuming without loss of generality
that $\wh a_\infty^{ij}=\delta_{ij}$,
we have the supersolution
\[\varphi(x)=|x|^{-1/2}\]
of $\wh a^{ij}(x)w_{ij}=0$ in $\ol B_{R_0}^\CC$ for some large $R_0>1$.
Since, for any $e\in\p B_1$,
\begin{equation}\label{eqn.vee}
\wh a^{ij}(v_{ee})_{ij}\geq0
\quad\mathrm{and}\quad
v_{ee}(x)\ra0~(|x|\ra\infty),
\end{equation}
we can use $\varphi$ as a barrier function,
as in \emph{Step 2} of the proof of \cref{thm.ELT-LEE-ND},
to conclude that
\[v_{ee}(x)\leq C\varphi(x)\leq C|x|^{-1/2},\]
for some constant $C>0$.
Let $\lambda_{\mathrm{max}}(M)$
and $\lambda_{\mathrm{min}}(M)$ denote
the maximal and minimal eigenvalue of the matrix $M$, respectively.
Then we have
\[\lambda_{\mathrm{max}}(D^2v)(x)\leq C|x|^{-1/2}.\]
On the other hand, since $\ol a^{ij}(x)v_{ij}=0$
and $\ol a^{ij}(x)$ is uniformly elliptic, we get
\[\lambda_{\mathrm{min}}(D^2v)(x)
\geq-C\lambda_{\mathrm{max}}(D^2v)(x)
\geq-C|x|^{-1/2}.\]
Hence we conclude that
\[\left|D^2u(x)-A\right|=\left|D^2v(x)\right|\leq C|x|^{-1/2},\]
which in turn implies
\[\left|\wh a^{ij}(x)-\wh a_\infty^{ij}\right|
\leq C|x|^{-1/2},~x\in B_1^\CC.\]
Thus the function
\[\wt\varphi(x)=2|x|^{2-n}-|x|^{2-n-1/4}\]
is a supersolution of $\wh a^{ij}(x)w_{ij}=0$
in $\ol B_{R_0}^\CC$ for some other large $R_0>1$.
Recalling \eqref{eqn.vee} and using $\wt\varphi$ as a barrier function,
as in \emph{Step 2} or rather \emph{Step 3}
of the proof of \cref{thm.ELT-LEE-ND}, we conclude that
\[v_{ee}(x)\leq C\wt\varphi(x)\leq C|x|^{2-n}.\]
Repeating the argument above, we get
\[\lambda_{\mathrm{max}}(D^2v)(x)\leq C|x|^{2-n},\]
\[\lambda_{\mathrm{min}}(D^2v)(x)
\geq-C\lambda_{\mathrm{max}}(D^2v)(x)
\geq-C|x|^{2-n},\]
and hence
\begin{equation}\label{eqn.D2u-A<CG}
\left|D^2u(x)-A\right|=\left|D^2v(x)\right|
\leq C|x|^{2-n}.
\end{equation}
Therefore,
\begin{equation}\label{eqn.olaij-Holder}
\left|\ol a^{ij}(x)-\ol a_\infty^{ij}\right|
\leq C|x|^{2-n},~x\in B_1^\CC,
\end{equation}
and
\begin{equation}\label{eqn.whaij-Holder}
\left|\wh a^{ij}(x)-\wh a_\infty^{ij}\right|
\leq C|x|^{2-n},~x\in B_1^\CC.
\end{equation}

\medskip

\emph{Step 2.}
We follow the argument in \cite[p. 567]{CL03}
to capture the linear and constant terms
in the asymptotic quadratic polynomial $Q(x)$ for the solution $u(x)$.
For any $e\in\p B_1$,
it follows from \eqref{eqn.D2u-A<CG} that
\[|Dv_e|\leq\left|D^2v(x)\right|
\leq C|x|^{2-n}\leq C|x|^{-1}.\]
Since $\wh a^{ij}(v_e)_{ij}=0$,
by \eqref{eqn.whaij-Holder} and \cref{cor.limit-finite.Dv},
we conclude that there exists a constant $b_e$ such that
\[v_e(x)=b_e+O(|x|^{2-n})~(|x|\ra\infty).\]
Let $b=(b_1,b_2,...,b_n)^T$ and
\[\ol v(x)=v(x)-b^Tx=u(x)-\left(\frac{1}{2}x^TAx+b^Tx\right).\]
Then
\begin{equation}\label{eqn.Du.n-2}
|Du(x)-(Ax+b)|=|D\ol v(x)|=O(|x|^{2-n})~(|x|\ra\infty).
\end{equation}
In particular,
\[\left|D\ol v(x)\right|\leq C|x|^{-1}.\]
Since $\ol a^{ij}\ol v_{ij}=\ol a^{ij}v_{ij}=0$,
by \eqref{eqn.olaij-Holder} and \cref{cor.limit-finite.Dv},
we thus deduce that there exists a constant $c$ such that
\[\ol v(x)=c+O(|x|^{2-n})~(|x|\ra\infty).\]
Let
\[Q(x)=\frac{1}{2}x^TAx+b^Tx+c.\]
Then
\[|u(x)-Q(x)|=|\ol v(x)-c|=O(|x|^{2-n})~(|x|\ra\infty).\]

\medskip

\emph{Step 3.}
For any fixed $x$ with $|x|$ sufficiently large,
let
\[E(y)=\left(\frac{2}{|x|}\right)^2(u-Q)\left(x+\frac{|x|}{2}y\right).\]
Then
\[\ul a^{ij}(y)E_{ij}(y)=F(A+D^2E(y))-F(A)=0,~y\in B_1,\]
where
\[\ul a^{ij}(y)=\int_0^1F_{M_{ij}}(A+tD^2E(y))dt.\]
By the Evans-Krylov estimate (fully nonlinear Schauder estimate)
and the Schauder estimate, we have
\[|D^kE(0)|\leq C_k\norm{E}_{L^\infty(B_1)}
\leq C_k|x|^{-n},~\text{for all}~k\in\N,\]
and hence
\[|D^k(u-Q)(x)|\leq C_k|x|^{2-n-k},~\text{for all}~k\in\N.\]

\medskip

\emph{Step 4.}
The uniqueness of the quadratic polynomial $Q(x)$
can be traced from the above argument.
Another way is the following.
Given the asymptotic behavior of $u(x)$ to $Q(x)$ near infinity,
the difference between any two quadratic asymptotics of the solution $u(x)$
is zero at infinity, in turn, they must be the same.
\end{proof}

\section{Proof of \cref{thm.u=Q+O.SLE}}\label{sec.EBT.SLE}

As in \cite{Y06,Y02}, we first make a transformation of the solution,
or a $U(n)$ rotation of the ambient space
$\C^n=\R^n\times \R^n\supset\set{(x,Du(x))}$,
so that the Hessian of the new potential function is bounded.
By symmetry, we only consider the case $\Theta>(n-2)\pi/2$.
Let $\sum_{i=1}^{n}\theta_i=(n-2)\pi/2+n\vth$
with $\theta_i=\arctan\lambda_i$ and $\vth\in(0,\pi/n)$.
Observe that
\[-\frac{\pi}{2}+n\vth<\theta_i<\frac{\pi}{2}.\]
The first inequality follows from $(n-2)\pi/2+n\vth<\theta_i+(n-1)\pi/2$,
and it enables us to extend $u$ smoothly over $\ol\Omega$ such that
\begin{equation}\label{eqn.D2u-lower-bd.semicv}
D^2u>-\cot(2\vth)I\geq(1-\cot\vth)I.
\end{equation}
We rotate the $(x,y)\in\R^n\times\R^n$ coordinate system
to $(\wt x,\wt y)$ by $\vth$, $\wt z=e^{-\sqrt{-1}\vth}z$,
namely $\wt x=\fc x+\fs y$ and $\wt y=-\fs x+\fc y$
with $(\fc,\fs)=(\cos\vth,\sin\vth)$.
Then $(x,Du(x))$ has a new parametrization
\begin{equation}\label{eqn.rotation-SLE}
\begin{cases}
\wt x=\fc x+\fs Du(x),\\
\wt y=-\fs x+\fc Du(x).
\end{cases}
\end{equation}
By the convexity of $u(x)+\cot\vth\,|x|^2/2$
from \eqref{eqn.D2u-lower-bd.semicv},
we have the distance increasing property
\begin{equation}\label{eqn.dist-incr}
|\wt x-\wt x^\ast|^2
=\sin^2\vth\,\big|\cot\vth\,x+Du(x)-\cot\vth\,x^\ast-Du(x^\ast)\big|^2
\geq\sin^2\vth\,|x-x^\ast|^2.
\end{equation}
We deduce that $x\mapsto\wt x$ is a diffeomorphism from $\R^n$ to $\R^n$
and $\wt\Omega=\wt x(\Omega)$ is a bounded domain
(for more details, see \emph{Step 1} of proofs
for \cref{thm.u=Q+O.QHE} and \cref{thm.u=Q+O.IHHE}).

Next we define the new potential
\begin{eqnarray}\label{eqn.wtu=xuDu.rotation}
\wt u(\wt x)
&=&\int^{\wt x}\wt y\cdot d\wt x
=\int^{x(\wt x)}\langle-\fs x+\fc Du(x),\fc dx+\fs dDu(x)\rangle\nonumber\\
&=&\frac{1}{2}\fc\fs\left(|Du(x(\wt x))|^2-|x(\wt x)|^2\right)
-\fs^2Du(x(\wt x))\cdot x(\wt x)+u(x(\wt x)),
\end{eqnarray}
where we integrated by parts for the last equality.
Note that the above two equivalent integrals are well-defined
for diffeomorphism $x\mapsto\wt x=\fc x+\fs Du(x)$.
It follows that $D\wt u(\wt x)=\wt y=-\fs x+\fc Du(x)$,
and by the chain rule
\begin{eqnarray*}
D^2\wt {u}
&=&(-\fs I+\fc D^2u)(\fc I+\fs D^2u)^{-1}\\
&=&\left(
\begin{matrix}
\tan(\theta_1-\vth)&~&~\\
~&\ddots&~\\
~&~&\tan(\theta_n-\vth)
\end{matrix}
\right)\quad\text{when $D^2u$ is diagonalized.}
\end{eqnarray*}
Therefore $\wt u$ satisfies
\[\sum_{i=1}^n\arctan\lambda_i(D^2\wt {u})=\frac{(n-2)\pi}{2}
\quad\mathrm{and}\quad
|D^2\wt u|<\cot\vth\quad\mathrm{in}\quad\R^n\setminus\wt\Omega.\]

\subsection{Proof of \cref{thm.u=Q+O.SLE} ($n\geq3$)}\label{subsec.EBT.SLE.hd}
\quad

\emph{Step 1.}
Now that $\wt u$ satisfies a uniformly elliptic fully nonlinear equation,
which is also concave by the convexity observation of the level set
\[\set{\lambda\in\R^n\,|\,\text{$\lambda$ satisfying \eqref{eqn.SLE}
with $\Theta=(n-2)\pi/2$}}\] \cite[Lemma 2.1]{Y06}.
Applying \cref{thm.u=Q+O.FNL} or \cref{lem.D2u->A.FNL} to $\wt u$, we obtain
\[D^{2}\wt u(\wt x)\ra\wt A\quad\mathrm{as}\quad|\wt x|\ra\infty\]
for some constant symmetric matrix $\wt A$.

\medskip

\emph{Step 2.}
We claim that
\begin{equation}\label{eqn.labd<cotvth}
\lambda_i(\wt A)<\cot\vth,~\text{for all}~i=1,2,...,n.
\end{equation}
Otherwise, by rotating the $\wt x$-space to make $\wt A$ diagonal,
we may assume that $\wt A_{11}=\cot\vth$.
Then the rotated graph $\set{(\wt x,D\wt u(\wt x))}$
would have the asymptote
\[\wt y_1=\p_1\wt u(\wt x)
=\cot\vth\;\wt x_1+\wt b_1+O(|\wt x|^{2-n})
=\cot\vth\;\wt x_1+O(1)\]
in $\set{(\wt x_1,\wt y_1)}\cap\R^{n}\setminus\ol{\wt\Omega}$,
according to the asymptotic behavior of $D\wt u$ by \cref{thm.u=Q+O.FNL}
(see also \eqref{eqn.Du.n-2}).
Thus we infer that
\begin{eqnarray*}
x_1&=&\wt x_1\cos\vth-\wt y_1\sin\vth\\
&=&\wt x_1\cos\vth-\wt x_1\cot\vth\sin\vth+O(1)=O(1),
\end{eqnarray*}
which states that $\R^n\setminus\ol{\Omega}$ is bounded in the $x_1$-direction
(geometrically, this also means that,
by rotating back to the original $(x_1,y_1)$-space,
the ``gradient'' graph $\set{(x,Du(x))}$ would be inside a vertical strip
of width $O(1)$ around the vertical $y_1$-axis),
a contradiction.

It follows from the above claim that
the matrix $\cos\vth\;I-\sin\vth\;\wt A$ is invertible.
By the explicit formula
\begin{equation}\label{eqn.D2u.ef}
D^2u(x)=\left(\sin\vth\;I+\cos\vth\;D^2\wt u(\wt x)\right)
\left(\cos\vth\;I-\sin\vth\;D^2\wt u(\wt x)\right)^{-1}
\end{equation}
resulting from \eqref{eqn.rotation-SLE}, we conclude that
\[D^2u(x)\ra A\quad(|x|\ra\infty)\]
with
\[A=\left(\sin\vth\;I+\cos\vth\;\wt A\right)
\left(\cos\vth\;I-\sin\vth\;\wt A\right)^{-1}.\]
Thus
\[|D^2u|\leq C(n,\Theta,u)<+\infty
\quad\mathrm{in}\quad\R^{n}\setminus\ol\Omega,\]
and hence the original equation \eqref{eqn.SLE} is also uniformly elliptic.
Applying \cref{thm.u=Q+O.FNL} to $u$,
we complete the proof of \eqref{eqn.SLE.u=Q+ON}.

\subsection{Proof of \cref{thm.u=Q+O.SLE} ($n=2$)}\label{subsec.EBT.SLE.2d}
\quad

\emph{Step 1.}
By rotation \eqref{eqn.rotation-SLE},
we have a harmonic function $\wt u$ satisfying
\[\Delta\wt u=0\quad\mathrm{and}\quad|D^2\wt u|\leq C(\Theta)
\quad\mathrm{in}\quad\R^2\setminus\ol{\wt\Omega}.\]
Set $z=\wt x_1+\sqrt{-1}\,\wt x_2$.
Then the holomorphic function
\[h(z)=\p_{\wt x_1}\wt u-\sqrt{-1}\,\p_{\wt x_2}\wt u\]
has linear growth at infinity.
By the Laurent expansion, we obtain
\begin{equation}\label{eqn.h(z).Laurent}
h(z)=a_{1}z+a_{0}+a_{-1}z^{-1}+a_{-2}z^{-2}+\cdots
\end{equation}
for all large $z$.
Since $\mathrm{Re}\int a_{-1}z^{-1}dz=\mathrm{Re}(a_{-1}\log z)$,
as a part of $\wt u$, is well defined in an exterior domain,
we see that $a_{-1}$ must be a real number.
Thus we have
\begin{equation}\label{eqn.Dwtu.expa}
D\wt u(\wt x)=D\wt Q(\wt x)+a_{-1}D\log{|\wt x|}+O(|\wt x|^{-2})
\quad\mathrm{as}\quad|\wt x|\ra\infty,
\end{equation}
where
\[\wt Q(\wt x)=\frac{1}{2}\wt x^T\wt A\wt x+\wt b^T\wt x\]
with
\begin{equation}\label{eqn.wtA&wtb.def}
\wt A=
\left(
\begin{array}{cc}
\mathrm{Re}\,a_1  & -\mathrm{Im}\,a_1 \\
-\mathrm{Im}\,a_1 & -\mathrm{Re}\,a_1 \\
\end{array}
\right)
\quad\text{and}\quad
\wt b=(\mathrm{Re}\,a_{0},-\mathrm{Im}\,a_{0}).
\end{equation}

Since the Laurent series \eqref{eqn.h(z).Laurent}
for holomorphic function $h(z)$
is allowed to be taken derivatives term by term,
it follows from \eqref{eqn.Dwtu.expa} that
\[D^2\wt u(\wt x)=:\wt A+O(|\wt x|^{-2})\quad\mathrm{as}\quad|\wt x|\ra\infty.\]

\medskip

\emph{Step 2.}
By \eqref{eqn.Dwtu.expa} and the same strip argument
in the proof of \eqref{eqn.labd<cotvth}, we deduce that
\[|\lambda_i(\wt A)|<\cot\vth\]
for $i=1,2$, where $\vth=\Theta/2$.
Thus the matrix $\cos\vth\,I-\sin\vth\,\wt A$ is invertible.
Recall $(\fc,\fs)=(\cos\vth,\sin\vth)$.
By the explicit formula \eqref{eqn.D2u.ef}, we obtain
\begin{eqnarray}\label{eqn.D2u=A+O}
D^2u(x)&=&(\fs I+\fc \wt A+O(|\wt x|^{-2}))
(\fc I-\fs \wt A+O(|\wt x|^{-2}))^{-1}\nonumber\\
&=&(\fs I+\fc \wt A)(\fc I-\fs \wt A)^{-1}+O(|\wt x|^{-2})\nonumber\\
&=:&A+O(|x|^{-2}),
\end{eqnarray}
where in the last equality we used the inequality $C|\wt x|\geq|x|$
resulting from the distance increasing inequality \eqref{eqn.dist-incr}.

Substituting the asymptotic behavior \eqref{eqn.Dwtu.expa} of $D\wt u$
into the inverse rotation formula of \eqref{eqn.rotation-SLE}, we get
\begin{subequations}
\begin{numcases}{}
x=\fc \wt{x}-\fs D\wt{u}(\wt{x})
=(\fc I-\fs \wt{A}-\fs a_{-1}|\wt x|^{-2}I)\wt{x}
-\fs \wt{b}+O(|\wt{x}|^{-2}),\qquad\label{eqn.x.ex}\\
Du(x)=\fs \wt{x}+\fc D\wt{u}(\wt{x})
=(\fs I+\fc \wt{A}+\fc a_{-1}|\wt x|^{-2}I)\wt{x}
+\fc \wt{b}+O(|\wt{x}|^{-2}).\qquad\label{eqn.Du.ex}
\end{numcases}
\end{subequations}

It follows from \eqref{eqn.x.ex} that
\begin{equation}\label{eqn.wtx.ex}
\wt{x}=(\fc I-\fs \wt{A}-\fs a_{-1}|\wt x|^{-2}I)^{-1}(x+\fs \wt{b})+O(|\wt{x}|^{-2}).
\end{equation}
Plugging \eqref{eqn.wtx.ex} into \eqref{eqn.Du.ex}, we obtain
\begin{eqnarray}\label{eqn.Du.log}
Du(x)&=&(\fs I+\fc \wt{A}+\fc a_{-1}|\wt x|^{-2}I)%
(\fc I-\fs \wt{A}-\fs a_{-1}|\wt x|^{-2}I)^{-1}(x+\fs \wt{b})
+\fc \wt{b}+O(|\wt{x}|^{-2})\nonumber\\
&=&\big[(\fs I+\fc \wt{A})(\fc I-\fs \wt{A})^{-1}
+(\fs I+\fc \wt{A})(\fc I-\fs \wt{A})^{-2}\fs a_{-1}|\wt x|^{-2}\nonumber\\
&~&\quad+(\fc I-\fs \wt{A})^{-1}\fc a_{-1}|\wt x|^{-2}\big](x+\fs \wt{b})
+\fc \wt{b}+O(|\wt{x}|^{-2})\nonumber\\
&=&Ax+(\fc I+\fs A)\wt{b}+a_{-1}(\fc +\fs A)(\fc I-\fs \wt{A})^{-1}%
x/|\wt x|^2+O(|\wt{x}|^{-2})\nonumber\\
&=&Ax+b+\frac{a_{-1}(I+A^2)x}{x^T(I+A^2)x}+O(|x|^{-2}),
\end{eqnarray}
where we used
\[1/|\wt x|^2=|(\fc I+\fs A)x+O(1)|^{-2}
=(x^T(\fc I+\fs A)^{2}x)^{-2}+O(|x|^{-3}),\]
and
\[(\cos\vth\,I-\sin\vth\,\wt A)^{-1}
=\cos\vth\,I+\sin\vth\,A
=\cos\big((\theta_1^\ast-\theta_2^\ast)/2\big)\,(I+A^2)^{1/2}\]
with $\theta_i^\ast=\arctan\lambda_i(A)$ for $i=1,2$.

Finally, by integrating \eqref{eqn.Du.log} term by term, we get
\begin{eqnarray}\label{eqn.SLE.u=abc+log+O}
u(x)&=&\frac{1}{2}x^TAx+b^Tx+c
+\frac{a_{-1}}{2}\log x^T(I+A^2)x+O(|x|^{-1})\nonumber\\
&=:&Q(x)+\frac{a_{-1}}{2}\log x^T(I+(D^2Q)^2)x+O(|x|^{-1})
\quad\mathrm{as}\quad|x|\ra\infty.\qquad
\end{eqnarray}

\medskip

\emph{Step 3.}
To calculate the coefficient $a_{-1}$ for the logarithmic term
\[\Gamma(x)=\frac{a_{-1}}{2}\log x^T(I+A^2)x,\]
as in \cite[p. 570]{CL03},
we integrate the algebraic form of equation \eqref{eqn.SLE}
\begin{equation}\label{eqn.SLE.2d.af}
\cos\Theta\,\Delta u+\sin\Theta\,\det D^2u=\sin\Theta.
\end{equation}
We have
\[\int_{E_R\setminus\Omega}\CS\,dx
=\int_{E_R\setminus\Omega}\CC\Delta u+\CS\det D^{2}u\,dx
=\int_{\p(E_R\setminus\Omega)}\CC u_{\nu}+\CS u_{1}(u_{22},-u_{12})\cdot\nu\,ds,\]
where $E_R=\set{x\in\R^2\,|\,x^T(I+A^2)x<R^2}$
and $(\CC,\CS)=(\cos\Theta,\sin\Theta)$.
In view of the asymptotic behaviors
\eqref{eqn.D2u=A+O} and \eqref{eqn.Du.log}, we get
\begin{eqnarray}\label{eqn.2pia-1}
&&\int_{\p\Omega}\CC u_{\nu}+\CS u_{1}(u_{22},-u_{12})\cdot\nu\,ds
+\int_{E_R\setminus\Omega}\CS\,dx\nonumber\\
&=&\int_{\p E_R}\CC(Q+\Gamma)_{\nu}+\CS(Q_{1}+\Gamma_{1})%
(Q_{22}+\Gamma_{22},-Q_{12}-\Gamma_{12})\cdot\nu\,ds+O(R^{-1})\nonumber\\
&=&\int_{\p E_R}\CC Q_{\nu}+\CS Q_{1}(Q_{22},-Q_{12})\cdot\nu\,ds\nonumber\\
&&\quad+\int_{\p E_R}\CC\Gamma_{\nu}+\CS(Q_{1}(\Gamma_{22},-\Gamma_{12})%
+\Gamma_{1}(Q_{22},-Q_{12}))\cdot\nu\,ds+O(R^{-1})\nonumber\\
&=&\int_{E_R}\CC\Delta Q+\CS\det D^{2}Q\,dx+2\pi a_{-1}+O(R^{-1}).
\end{eqnarray}
Letting $R$ go to $\infty$, we obtain
\begin{eqnarray*}
a_{-1}=\frac{1}{2\pi}\left(\int_{\p\Omega}\CC u_{\nu}
+\CS u_{1}(u_{22},-u_{12})\cdot\nu\,ds-\CS|\Omega|\right)=d.
\end{eqnarray*}

We still have to verify the appearance
of the $2\pi a_{-1}$ term in \eqref{eqn.2pia-1}.
Instead of going through the direct, but tricky and long calculation
for the corresponding boundary integral, we use divergence theorem.
Without loss of generality,
we assume $A$ is diagonal with eigenvalues $(\mu_{1},\mu_{2})$.
From the equation $\arctan\mu_{1}+\arctan\mu_{2}=\Theta$, it follows that
\begin{eqnarray}\label{eqn.delta-argument}
&&\cos\Theta\,\Delta\Gamma
+\sin\Theta\,(Q_{22}\Gamma_{11}-2Q_{12}\Gamma_{12}+Q_{11}\Gamma_{22})\nonumber\\
&=&\sqrt{(1+\mu_{1}^{2})(1+\mu_{2}^{2})}
\left(\frac{1}{1+\mu_{1}^{2}}\Gamma_{11}+\frac{1}{1+\mu_{2}^{2}}\Gamma_{22}\right)
\nonumber\\
&=&2\pi a_{-1}\delta_{0}\quad\text{in}\quad\mathbb{R}^{2}.
\end{eqnarray}
Hence%
\begin{eqnarray*}
&&\int_{\partial E_R}\CC\Gamma_\nu
+\CS\big(Q_{1}(\Gamma_{22},-\Gamma_{12})+\Gamma_{1}(Q_{22},-Q_{12})\big)\cdot\nu\,ds\\
&=&\int_{E_R}\cos\Theta\,\Delta\Gamma
+\sin\Theta\,(Q_{22}\Gamma_{11}-2Q_{12}\Gamma_{12}+Q_{11}\Gamma_{22})\,dx
=2\pi a_{-1}.
\end{eqnarray*}

\medskip

\emph{Step 4.}
For any fixed $x$ with $|x|$ sufficiently large,
let
\[E(y)=\left(\frac{2}{|x|}\right)^2(u-Q-\Gamma)\left(x+\frac{|x|}{2}y\right)
\quad\mathrm{and}\quad
\ol\Gamma(y)=\left(\frac{2}{|x|}\right)^2\Gamma\left(x+\frac{|x|}{2}y\right).\]
Then
\[\ol a^{ij}(y)(E+\ol\Gamma)_{ij}(y)
=F(A+D^2E(y)+D^2\ol\Gamma(y))-F(A)=0,~y\in B_1,\]
where
\[\ol a^{ij}(y)=\int_0^1F_{M_{ij}}(A+t(D^2E(y)+D^2\ol\Gamma(y)))dt,\]
and
\begin{equation}\label{eqn.F(M).2d-SLE}
F(M)=\arctan\lambda_1(M)+\arctan\lambda_2(M).
\end{equation}

By the Nirenberg estimate (two dimensional fully nonlinear Schauder estimate)
and the Schauder estimate, we have
\begin{eqnarray*}
|D^kE(0)|
&\leq& C_k(\norm{E}_{L^\infty(B_1)}+\norm{\ol a^{ij}\ol\Gamma_{ij}}_{C^\alpha(B_1)})\\
&\leq& C_k(\norm{E}_{L^\infty(B_1)}+|x|^{-2})
\leq C_k|x|^{-1},~\text{for all}~k\in\N.
\end{eqnarray*}
Therefore
\[|D^k(u-Q-\Gamma)(x)|\leq C_k|x|^{-k-1},~\text{for all}~k\in\N.\]

\medskip

\emph{Step 5.}
The uniqueness of the quadratic polynomial $Q(x)$
is proved in the same way as in \emph{Step 4}
in the proof of \cref{thm.u=Q+O.FNL}.

\begin{remark}\label{rmk.2d-log-linearization}
Once we reach a fast enough rate
for the Hessian of the solution approaching its limit at infinity,
we can reveal the asymptotics of the solution through linearized equations.
By the asymptotic behavior \eqref{eqn.D2u=A+O} of $D^2u$,
we set $v(x)=u(x)-x^TAx/2$.

(i)
From the original trigonometric form equation \eqref{eqn.SLE}, we have
\[\ul a^{ij}(x)v_{ij}(x)=F(A+D^2v(x))-F(A)=0,~x\in \ol B_1^\CC,\]
where $F$ is the one given in \eqref{eqn.F(M).2d-SLE} and
$\ul a^{ij}(x)=\int_0^1F_{M_{ij}}(A+tD^2v(x))dt$.
Since \eqref{eqn.D2u=A+O} reads
\[|D^2v(x)|=O(|x|^{-2})~(|x|\ra\infty),\]
it follows that
\[\mathrm{tr}\left((I+A^2)^{-1}D^2v\right)=F_{M_{ij}}(A)v_{ij}
=\left(F_{M_{ij}}(A)-\ul a^{ij}\right)v_{ij}=O(|x|^{-4})~(|x|\ra\infty).\]
From the Newtonian representation of $v(x)$ as in \cite[p. 569]{CL03},
we deduce that
\[u(x)=\frac{1}{2}x^{T}Ax+v(x)=\frac{1}{2}x^{T}Ax+b^Tx+c+\frac{d}{2}\log x^T(I+A^2)x
+O(|x|^{-1})~(|x|\ra\infty).\]
for some $b\in\R^n$ and $c,d\in\R$.

(ii)
Another way is to consider the algebraic form
of the special Lagrangian equation \eqref{eqn.SLE.2d.af}.
It follows from \eqref{eqn.D2u=A+O}
and the expansion formula of the determinant that
\[\mathrm{tr}\left(\mathcal{M}D^2v\right)
=O(|x|^{-4})~(|x|\ra\infty),\]
where $\mathcal{M}=\cos\Theta\,I+\sin\Theta\,(\det A)A^{-1}$.
Because of (the ``conformality'')
\[\mathcal{M}^{-1}=(\cos\Theta\,I_2+\sin\Theta\,(\det{A})A^{-1})^{-1}
=\frac{1}{\sqrt{\det(I+A^2)}}(I+A^{2}),\]
we obtain the same linearized equation
and the same logarithmic term $\log x^T\mathcal{M}^{-1}x$.

\medskip

Note that, when $\Theta=\pi/2$,
the special Lagrangian equation \eqref{eqn.SLE.2d.af}
becomes the Monge-Amp\`{e}re equation $\det D^2u=1$.
We have $\mathcal{M}=A^{-1}$
and hence the logarithmic term $\log x^TAx$,
which is as same as the one given
in \cite[Theorem 1.2]{CL03} (see also \cref{thm.u=Q+O.MAE}).
\end{remark}

\begin{remark}\label{rmk.2d-log-representation}
Via the rotation argument in Step 1,
we actually have a harmonic representation of the potential $u$
to two dimensional special Lagrangian equations,
which in turn, also leads to the asymptotics of the solution.

Write $z=\wt x_1+\sqrt{-1}\,\wt x_2$, $\wt x=(\wt x_1,\wt x_2)$
and $\wt y=(\wt u_{\wt x_1},\wt u_{\wt x_2})$.
From \eqref{eqn.h(z).Laurent}, we have
\begin{eqnarray*}
u(x)&=&\int^x Du\cdot dx
=\int^x(\fs\wt{x}+\fc\wt{y})\cdot d(\fc\wt{x}-\fs\wt{y})\\
&=&\int^x \fs\fc(\wt{x}d\wt{x}-\wt{y}d\wt{y})+\fc^{2}%
\wt{y}d\wt{x}-\fs^{2}\wt{x}d\wt{y}\\
&=&\int^x\frac{1}{2}\fs\fc d(|\wt{x}|^{2}-|\wt{y}|^{2})
+\mathrm{Re}(\fc^{2}hdz-\fs^{2}zh_{z}dz)\\
&=&\frac{1}{2}\fs\fc(|z|^{2}-|h|^{2})
+\mathrm{Re}\int^z(\fc^{2}h-\fs^{2}zh_{z})dz\\
&=&\frac{1}{2}\fs\fc(|z|^{2}-|h|^{2})
+\frac{1}{2}(\fc^{2}-\fs^{2})\mathrm{Re}(a_{1}z^{2})\\
&&\qquad+\fc^2\mathrm{Re}(a_0z)+a_{-1}\log|z|+O(|z|^{-1}).
\end{eqnarray*}
Since $\wt x=(\cos\vth\,I+\sin\vth\,A)x+O(1)$
via the asymptotic behavior \eqref{eqn.Du.log} of $Du$
(the rough version is enough), we obtain
\begin{eqnarray*}
\log|z|^2
&=&\log x^T(\cos\vth\,I+\sin\vth\,A)^2x+O(1)\\
&=&\log x^T(I+A^2)x+O(1).
\end{eqnarray*}
On the other hand,
since $x=(\fc I-\fs\wt{A})\wt x+O(1)$ via \eqref{eqn.x.ex},
and $A=(\fs I+\fc\wt{A})(\fc I-\fs\wt{A})^{-1}$,
by the definition \eqref{eqn.wtA&wtb.def} of $\wt A$,
it is not hard to verify that the highest degree term of
\[\fs\fc(|z|^{2}-|a_1z|^{2})+(\fc^{2}-\fs^{2})\mathrm{Re}(a_{1}z^{2})\]
is exactly $x^{T}Ax$.
Thus we conclude also that
\[
u(x)=\frac{1}{2}x^{T}Ax+b^Tx+c+\frac{d}{2}\log x^T(I+A^2)x
+O(|x|^{-1})~(|x|\ra\infty).\]

For a complex representation of the solution
to the two dimensional Monge-Amp\`{e}re equation
via the Legendre-Lewy transformation, see \cite{FMM96,FMM99}.
\end{remark}

\section{Further perspectives}\label{sec.FP}

In 2003, Caffarelli and Li \cite{CL03} extended
the J\"{o}rgens-Calabi-Pogorelov theorem for Monge-Amp\`{e}re equations
(which asserts that every convex global solution of the Monge-Amp\`{e}re equation
must be a quadratic polynomial) to exterior domains as follows.
\begin{theorem}\label{thm.u=Q+O.MAE}
(see \cite{CL03})
Let $u$ be a smooth convex solution of the Monge-Amp\`{e}re equation
$\det D^2u=1$ in the exterior domain $\R^n\setminus\ol{\Omega}$.
Then there exists a unique quadratic polynomial $Q(x)$
such that when $n\geq3$,
\begin{equation}\label{eqn.MAE.u=Q+ON}
u(x)=Q(x)+O_k(|x|^{2-n})\quad\mathrm{as}\quad|x|\ra\infty
\end{equation}
for all $k\in\N$, and when $n=2$,
\begin{equation}\label{eqn.MAE.u=Q+log+ON}
u(x)=Q(x)+\frac{d}{2}\log{x^TD^2Qx}+O_k(|x|^{-1})
\quad\mathrm{as}\quad|x|\ra\infty
\end{equation}
for all $k\in\N$, where
\[d=\frac{1}{2\pi}\left(\int_{\p\Omega}
u_{1}(u_{22},-u_{12})\cdot\nu\,ds-|\Omega|\right).\]
\end{theorem}
Case $n=2$ was treated earlier in 1999
by Ferrer, Mart\'{\i}nez and Mil\'{a}n \cite{FMM99}
using complex analysis method.
The method in \cite{CL03} is, to deduce first that
the solution is close to a quadratic polynomial
with a sub-quadratic error,
by using Caffarelli's theory on Monge-Amp\`{e}re equation,
and then to obtain the H\"{o}lder closedness at infinity
of the Hessian to some constant matrix,
by estimates of Pogorelov, Evans-Krylov, and Schauder.

We remark that, by extending the solution $u$ inside $\Omega$
such that the new $u$ is smooth and convex in $\R^n$,
and then invoking the Pogorelov estimate
(see \cite[pp. 73--76]{P78} or \cite[pp. 467--471]{GT98}),
we deduce that
\[\norm{D^2u}_{L^\infty\left(\R^n\setminus\ol{\Omega}\right)}
\leq C(n,u,\Omega)<+\infty,\]
which also implies that the equation is uniformly elliptic.
Applying \cref{thm.u=Q+O.FNL}, we have a new proof
of the above exterior J\"{o}rgens-Calabi-Pogorelov type result
for the Monge-Amp\`{e}re equation when $n\geq3$.
Our argument for \cref{thm.u=Q+O.SLE} ($n=2$) gives yet another proof
of \cref{thm.u=Q+O.MAE} ($n=2$), as the Monge-Amp\`{e}re equation
now is equivalent to the special Lagrangian equation
\eqref{eqn.SLE} with $\Theta=\pi/2$.

\bigskip

In 2010, Chang and the third author \cite{CY10} proved
an entire Liouville theorem for the quadratic Hessian equation,
which asserts that every convex solution must be quadratic.
The argument is to make a Legenre-Lewy transformation of the solution
to a new solution of a new uniformly elliptic and convex equation
with bounded Hessian from both sides,
so that Evans-Krylov-Safonov theory applies.
Combining this idea with our exterior Liouville \cref{thm.u=Q+O.FNL} for
general fully nonlinear elliptic and convex equations,
we obtain the following exterior Liouville theorem for the quadratic Hessian equation.
\begin{theorem}\label{thm.u=Q+O.QHE}
Let $n\geq3$ and let $u$ be a smooth solution of the quadratic Hessian equation
$\sigma_2(\lambda(D^2u))=1$ in the exterior domain $\R^n\setminus\ol{\Omega}$.
Suppose
\[D^2u>\left(\delta-\sqrt{2/(n(n-1))}\right)I
\quad\mathrm{in}\quad\R^n\setminus\ol{\Omega}\]
for any fixed $\delta>0$.
Then there exists a unique quadratic polynomial $Q(x)$ such that
\[u(x)=Q(x)+O_k(|x|^{2-n})\quad\mathrm{as}\quad|x|\ra\infty\]
for all $k\in\N$.
\end{theorem}

\begin{proof}
\emph{Step 1.}
As in \cite{CY10}, we make a Legendre-Lewy transformation of the solution
to a solution of a new uniformly elliptic and convex equation
with bounded Hessian from both sides.

Write $K=\sqrt{2/(n(n-1))}$
and let $w(x)=u(x)+K|x|^2/2$ for all $x\in\R^n\setminus\ol{\Omega}$.
Then $D^2w>\delta I$ in $\R^n\setminus\ol{\Omega}$.
By assuming $\p\Omega$ is smooth and extending $u$ smoothly to $\R^n$
such that $D^2w>\delta I$ in $\R^n$, we have the distance increasing property
\[|Dw(x)-Dw(x_\star)|
=\abs{\int_0^1D^2w(x_\star+t(x-x_\star))(x-x_\star)dt}
\geq\delta|x-x_\star|\]
for all $x,x_\star\in\R^n$.
Thus $x\mapsto y=Dw(x)$ is globally injective.
Because the Jacobian of the map $\det D_xy=\det D^2w(x)\neq0$,
the closed map $Dw(x)$ is also open.
Therefore, $Dw(x)$ is surjective, $Dw(\R^n)=\R^n$,
$Dw(\Omega)=:\wt\Omega$ is a bounded domain,
and hence $y\mapsto x$, $\R^n\ra\R^n$ is also bijective.

Consider the Legendre transform $\ol w(y)$ of $w(x)$ given by
$\ol w(y)=x(y)\cdot y-w(x(y))$.
We have $x=D\ol w(y)$ and $D^2\ol w(y)=(D^2w(x))^{-1}$.
It follows that the function $\wt u(y)=-\ol w(y)$ satisfies
\begin{equation}\label{eqn.D2wtu.QHE}
-\delta^{-1}I<D^2\wt u(y)=-\left(D^2u(x)+KI\right)^{-1}<0
\end{equation}
for all $y\in\R^n$ and
\[g\left(\wt\lambda(D^2\wt u)\right)
=\sigma_2\left((-\wt\lambda_1^{-1}-K,...,-\wt\lambda_n^{-1}-K)\right)
=1\quad\mathrm{in}\quad\wt\Omega^\CC.\]

As proved in \cite[pp. 661--663]{CY10}, we have
\begin{enumerate}[(i)]
\item
the level set $\Sigma=\set{\wt\lambda\,|\,g(\wt\lambda)=1}$ is convex;

\item
the normal vector $Dg$ of the level set $\Sigma$
is uniformly inside the positive cone
$\Gamma^+=\big\{\wt\lambda\,\big{|}\,\wt\lambda_i>0
~\text{for all}~i=1,2,...,n\big\}$
provided $\wt\lambda_i\in(-\delta^{-1},0)$ for all $i=1,2,...,n$.
\end{enumerate}
Thus $\wt u(y)$ satisfies a uniformly elliptic equation with convexity.

\medskip

\emph{Step 2.}
In view of \eqref{eqn.D2wtu.QHE} and applying \cref{thm.u=Q+O.FNL}, we obtain
\[\wt u(y)=\frac{1}{2}y^{T}\wt Ay+\wt b^Ty+\wt c+O_k(|y|^{2-n})\]
as $|y|\ra\infty$, for all $k\in\N$.
In particular,
\[D^2\wt w(y)\ra\wt A
\quad\mathrm{and}\quad
x=D\wt w(y)=-D\wt u(y)=-\wt Ay+O(1)\]
as $|y|\ra\infty$.
By the strip argument described in the proof
of \eqref{eqn.labd<cotvth} in \cref{subsec.EBT.SLE.hd},
we see that $\lambda_i(\wt A)<0$ for all $i=1,2,...,n$.
(Otherwise, $\R^n\setminus\Omega$ is bounded
in $x_{i_0}$-direction for some $i_0$, a contradiction.)
Thus the matrix $\wt A$ is invertible.
Therefore
\[D^2u(x)=-(D^2\wt u(y))^{-1}-KI\ra-\wt A^{-1}-KI=:A\]
as $|x|\ra\infty$, and $|D^2u(x)|\leq C$ for all $x\in\Omega^\CC$,
which implies that
the original quadratic Hessian equation $\sigma_2(\lambda)=1$
is uniformly elliptic in $D^2u(\Omega^\CC)$.
Since the level set $\set{\lambda\,|\,\sigma_2(\lambda)=1}$ is originally convex,
by applying \cref{thm.u=Q+O.FNL} again,
we thus complete the proof of \cref{thm.u=Q+O.QHE}.
\end{proof}

\bigskip

In 1960, Flanders \cite{Fl60} established an entire Liouville theorem
for the inverse harmonic Hessian equation
\begin{equation}\label{eqn.IHHE}
\frac{1}{\lambda_1(D^2u)}+\cdots+\frac{1}{\lambda_n(D^2u)}=1,
\end{equation}
which says that every smooth convex solution $u$ of \eqref{eqn.IHHE}
in the whole space $\R^n$ must be a quadratic polynomial.
As an application of the same idea in establishing our main \cref{thm.u=Q+O.SLE},
we obtain the following exterior Liouville theorem
for inverse harmonic Hessian equations.
\begin{theorem}\label{thm.u=Q+O.IHHE}
Let $u$ be a smooth convex solution
of the inverse harmonic Hessian equation \eqref{eqn.IHHE}
in the exterior domain $\R^n\setminus\ol{\Omega}$.
Then there exists a unique quadratic polynomial $Q(x)$ such that
when $n\geq3$,
\begin{equation}\label{eqn.IHE.u=Q+ON}
u(x)=Q(x)+O_k(|x|^{2-n})\quad\mathrm{as}\quad|x|\ra\infty
\end{equation}
for all $k\in\N$, and when $n=2$,
\begin{equation}\label{eqn.IHE.u=Q+log(A^2)+ON}
u(x)=Q(x)+\frac{d}{2}\log{x^T(D^2Q)^2x}+O_k(|x|^{-1})
\quad\mathrm{as}\quad|x|\ra\infty
\end{equation}
for all $k\in\N$, where
\[d=\frac{1}{2\pi}\int_{\p\Omega}u_1(u_{22},-u_{12})\cdot\nu-u_\nu\,ds.\]
\end{theorem}

\begin{proof}
We first make a Legendre transform of the solution $u$
to a new solution $\ol u$ to the Laplace equation
with $D^2\ol u$ being bounded from both sides.

From the equation \eqref{eqn.IHHE},
it is clear that $D^2u>I$ in $\ol\Omega^\CC$.
By assuming $\p\Omega$ is smooth and extending $u$ smoothly to $\R^n$
such that $D^2u>I$ in $\R^n$, we have the distance increasing property
\[|Du(x)-Du(x_\star)|
=\abs{\int_0^1D^2u(x_\star+t(x-x_\star))(x-x_\star)dt}
\geq|x-x_\star|\]
for all $x,x_\star\in\R^n$.
Thus $x\mapsto y=Du(x)$ is globally injective.
Because the Jacobian of the map $\det D_xy=\det D^2u(x)\neq0$,
the closed map $Du(x)$ is also open.
Therefore, $Du(x)$ is surjective, $Du(\R^n)=\R^n$,
and $Du(\Omega)=:\wt\Omega$ is a bounded domain.
Then
\[\ol u(y)=\int^y x\cdot dDu(x)
=x\cdot Du(x)-\int^{x(y)} Du(x)\cdot dx
=x(y)\cdot y-u(x(y)),\]
leading to the Legendre transform of $u$.
Note that the above two equivalent integrals
are well defined for diffeomorphism $x\mapsto y=Du(x)$.
It follows that $x=D\ol u(y)$,
and by the chain rule, $D^2\ol u(y)=(D^2u(x))^{-1}$.
Thus
\[\Delta\ol u=1 \quad\mathrm{and}\quad 0<D^2\ol u<I
\quad\mathrm{in}\quad \R^n\setminus\wt\Omega.\]

\medskip

\textbf{Case $\mathbf{n\geq3}$}.\quad
Invoking \cref{thm.u=Q+O.FNL}
(the proof is much simpler, as the equation now is Laplace), we have
\[\ol u(y)=\frac{1}{2}y^{T}\ol Ay+\ol b^Ty+\ol c+O_k(|y|^{2-n})
\quad\mathrm{as}\quad|y|\ra\infty\]
for all $k\in\N$.
In particular,
\[x=D\ol u(y)=\ol Ay+\ol b+O_k(|y|^{1-n})=\ol Ay+O(1).\]
By the strip argument, as described in the proof
of \eqref{eqn.labd<cotvth} in \cref{subsec.EBT.SLE.hd},
we see that the matrix $\ol A$ is invertible.
Hence
\begin{equation}\label{eqn.D2u(x)->A.IHHE}
D^2u(x)=\left(D^2\ol u(y)\right)^{-1}\ra\ol A^{-1}=:A
\quad\mathrm{and}\quad D^2u(x)=A+O_k(|x|^{-n})
\end{equation}
as $|x|\ra\infty$, and
\[|D^2u(x)|\leq C,~\text{for all}~x\in\Omega^\CC.\]
Applying \cref{thm.u=Q+O.FNL} again, we finally obtain
\[u(x)=Q(x)+O_k(|x|^{2-n})\quad\mathrm{as}\quad|x|\ra\infty\]
for all $k\in\N$.

\begin{remark}
Another way to reach the above asymptotic behavior
\eqref{eqn.IHE.u=Q+ON} is to adopt a similar,
but simpler (without logarithmic term) substitution procedure as in the
proof of \eqref{eqn.SLE.u=Q+log+ON} in \cref{subsec.EBT.SLE.2d}.
Indeed, by substituting
\[y=Ax-A\ol b+O_k(|x|^{1-n})\]
into
\[u(x)=x\cdot y-\ol u(y)=x\cdot y-\ol Q(y)+O_k(|y|^{2-n}),\]
we obtain \eqref{eqn.IHE.u=Q+ON}.
Noting that \eqref{eqn.D2u(x)->A.IHHE} reads $D_xy=D^2u(x)=A+O_k(|x|^{-n})$,
by the chain rule we see that the asymptotic behavior $O_k(|x|^{2-n})$
for any $k$ is also preserved.
\end{remark}

\textbf{Case $\mathbf{n=2}$}.\quad
Now we are exactly in a similar situation as in
\cref{subsec.EBT.SLE.2d} for the proof of \eqref{eqn.SLE.u=Q+log+ON}
of \cref{thm.u=Q+O.SLE} ($n=2$).
Repeating the complex analysis argument of \emph{Step 1},
the similar, but simpler notation-wise rotation argument of \emph{Step 2}
(the Legendre transform is just a $\pi/2$-$U(n)$
rotation followed by a conjugation,
namely, \eqref{eqn.wtu=xuDu.rotation} with $(\fc,\fs)=(0,1)$),
the same divergence argument of \emph{Step 3},
and the same Schauder argument of \emph{Step 4} in \cref{subsec.EBT.SLE.2d},
we conclude that
\[u(x)=Q(x)+\frac{d}{2}\log{x^T(D^2Q)^2x}+O_k(|x|^{-1})
\quad\mathrm{as}\quad|x|\ra\infty\]
for all $k\in\N$.
In particular, the boundary representation for $d$
is similarly calculated via integration of
the algebraic form of the inverse Harmonic Hessian equation
\[\Delta u-\det D^{2}u=0\]
and a corresponding ``$\delta$-function'' argument to \eqref{eqn.delta-argument}.

The uniqueness of $Q(x)$ is proved in the same way as in \emph{Step 4}
in the proof of \cref{thm.u=Q+O.FNL}.

\begin{remark}
Let $w(x)=u(x)-|x|^2/2$.
Then equation \eqref{eqn.IHHE} is equivalent
to the two dimensional Monge-Amp\`{e}re equation $\det D^2w=1$
and also the two dimensional special Lagrangian equation with $\Theta=\pi/2$.
From \eqref{eqn.SLE.u=Q+log+ON} in \cref{thm.u=Q+O.SLE}
and $(D^2Q)^2=(\det D^2Q)(D^2Q-I)$,
\eqref{eqn.IHE.u=Q+log(A^2)+ON} follows.

Note that one can also proceed as
in \cref{rmk.2d-log-linearization} or \cref{rmk.2d-log-representation}
to obtain \eqref{eqn.IHE.u=Q+log(A^2)+ON}.
\end{remark}
\end{proof}


\noindent\textbf{Acknowledgments.}\quad
The first and the second authors are partially supported by NSFC. 11671316.
The third author is partially supported by an NSF grant.
The second author thanks Dr. Yongpan Huang
for several useful discussions in the early stage of this paper.


\end{document}